\documentclass{amsart}
\usepackage{latexsym,amsmath,amssymb}
\usepackage{cite}
\newtheorem{thm}{Theorem}[section]
\newtheorem{cor}[thm]{Corollary}
\newtheorem{exam}[thm]{Example}

\newtheorem{prop}[thm]{Proposition}
\theoremstyle{definition}
\theoremstyle{remark}

\numberwithin{equation}{section}

\begin{document}

\title[]
{Isometric and Quasi-isometric weighted composition operators}

\author{\sc\bf M. S. Al Ghafri, Y. Estaremi and M. Z. Gashti}
\address{\sc M. S. Al Ghafri}
\email{mohammed.alghafri@utas.edu.om}
\address{Department of Mathematics, University of Technology and Applied Sciences, Rustaq {329},  Oman.}
\address{\sc Y. Estaremi}
\email{y.estaremi@gu.ac.ir}
\address{Department of Mathematics, Faculty of Sciences, Golestan University, Gorgan, Iran.}

\address{\sc M. Z. Gashti}
\email{gashti@ieee.org}
\address{German Informatics Society, Hamburg, Germany.}

\thanks{}

\thanks{}



\date{}

\dedicatory{}

\commby{}

\begin{abstract}
In this paper we characterize $m$-isometric and quasi-$m$-isometric weighted composition operators on the Hilbert space $L^2(\mu)$. Also, we find that normal-$m$-isometry and normal quasi-$m$-isometry weighted composition operators have finite spectrum. Consequently we have the results for composition and multiplication operators. In addition, we prove that for $m\geq 2$, a multiplication operator is $m$-isometry (quasi-$m$-isometry) if and only if it is $2$-isometry (quasi-$2$-isometry). Some examples are provided to illustrate our results.

\end{abstract}

\maketitle

\section{ \sc\bf Introduction and Preliminaries}

Let  $\mathcal{B}(\mathcal{H})$ be the Banach algebra of bounded linear operators acting on an infinite
dimensional complex Hilbert space $\mathcal{H}$. An operator $T\in \mathcal{B}(\mathcal{H})$ is called $m$-isometry if 
$$B_m=\sum_{k=0}^{m}(-1)^{m-k}\binom{m}{k}T^{*^{k}}T^{k}=0$$
and also $T$ is called quasi-m-isometry if
 $$\sum_{k=0}^{m}(-1)^{m-k}\binom{m}{k}T^{*^{k+1}}T^{k+1}=T^*B_mT=0.$$ 
 Specially, the operator $T$ is quasi-2-isometry if 
$$T^{*^3}T^3-2T^{*^2}T^2+T^*T=0.$$

It is well known that for all $m\geq 2$,
$$\{\text{isometric operators}\} \subseteq \{\text{2-isometric operators}\} \subseteq \{\text{m-isometric operators}\}.$$
Also, $T$ is called quasi-isometry if $T^{*^2}T^2=T^*T$. Clearly every
quasi-2-isometric operator is quasi-$m$-isometric (\cite{mp1}, Theorem 2.4). Moreover, the
classes of quasi-isometric operators and 2-isometric operators are sub-classes of the
class of quasi-2-isometric operators. In fact, inclusions are proper. There are some examples of quasi-2-isometric operators which are not quasi-isometric, one can see (\cite{mp1}, Example 1.1). For more details on $m$-isometric operators one can see \cite{as3, ber, bj,bm,cot,jl,je,mp} and references therein.\\
In this paper we are going to characterize $m$-isometric and quasi-$m$-isometric weighted composition operators on the Hilbert space $L^2(\mu)$.

\section{ \sc\bf Main Results}

Let $(X, \mathcal{F}, \mu)$ be a measure space, $\varphi:X\rightarrow X$ be a non-singular measurable transformation (we mean $\mu\circ\varphi^{-1}\ll\mu$), $h=\frac{d\mu\circ\varphi^{-1}}{d\mu}$ (the Radon-Nikodym derivative of the measure $\mu\circ\varphi^{-1}$ with respect to the measure $\mu$), $Ef=E(f|\varphi^{-1}(\mathcal{F}))$ (the conditional expectation of $f$ with respect to the sigma subalgebra $\varphi^{-1}(\mathcal{F})$ of $\mathcal{F}$) and $C_{\varphi}f=f\circ\varphi$, for every $f\in L^2(\mu)$, be the related composition operator on $L^2(\mu)$.
We recall that for every $f\in L^2(\mu)$,
\begin{itemize}
 \item $C^*_{\varphi}f=hE(f)\circ\varphi^{-1}$
 \item $C^*_{\varphi}C_{\varphi}f=h.f$,
 \item $C_{\varphi}C^*_{\varphi}f=(h\circ\varphi)Ef$
 \item $|C_{\varphi}|f=\sqrt{h}.f$.
\end{itemize}
And let $W=M_uC_{\varphi}$ be the weighted composition operator on $L^2(\mu)$, in which $u:X\rightarrow \mathbb{C}$ is a positive measurable function and $P$ is the orthogonal projection of $L^2(\mu)$ onto $\overline{\mathcal{R}(W)}$. Then we have
\begin{itemize}
 \item $W^*f=hE(u.f)\circ\varphi^{-1}$,
 \item $W^*Wf=hE(u^2)\circ\varphi^{-1}.f$,
 \item $WW^*f=J\circ\varphi P(f)=u.(h\circ\varphi)E(uf)$,
 \item $|W|f=\sqrt{hE(u^2)\circ\varphi^{-1}}.f$.
\end{itemize}

For $n\in \mathbb{N}$, let $h_n=\frac{d\mu\circ\varphi^{-n}}{\mu}$, $u_n=\prod^{n-1}_{k=0}u\circ\varphi^k$ and 
$$J_n=h_nE_n(|u_n|^2)\circ\varphi^{-n},$$
in which $E_n$ is the conditional expectation operator with respect to the sigma subalgebra $\varphi^{-n}(\mathcal{F})$ of $\mathcal{F}$. In particular we denote $J=J_1$ and $h=h_1$. Since $W^n(f)=u_nf\circ\varphi^n$, for every $n\in\mathbb{N}$ and $f\in L^2(\mu)$, then $W^n$, for $n\geq 2$, is also a weighted composition operator and we have 
\begin{align*}
\int_XJ_n|f|^2d\mu&=\|W^nf\|^2\\
&=\int_X|u_n|^2|f|^2\circ\varphi^nd\mu\\
&\int_Xh_{n-1}E_{n-1}(|u_n|^2|f|^2\circ\varphi^n)\circ\varphi^{-(n-1)}d\mu\\
&\int_Xh_{n-1}E_{n-1}(|u_{n-1}|^2)\circ\varphi^{-(n-1)}|u|^2|f|^2\circ\varphi d\mu\\
&\int_XJ_{n-1}|u|^2|f|^2\circ\varphi d\mu\\
&\int_XhE(J_{n-1}|u|^2)\circ\varphi^{-1}|f|^2d\mu.
\end{align*}
 This implies that
  $$J_n=hE(J_{n-1}|u|^2)\circ\varphi^{-1},$$
 and so for $n\geq 2$ if $u\equiv 1$, then 
 $$h_n=hE(h_{n-1}|u|^2)\circ\varphi^{-1}.$$ For more information about composition and weighted composition operators on $L^2(\mu)$, one can see \cite{bbl, cj}.\\

Now by using basic properties of weighted composition operators, as mentioned above, we characterize $m$-isometric and quasi-$m$-isometric weighted composition operators on $L^2(\mu)$.

\begin{thm}\label{t1.6}
Let $W=uC_{\varphi}$ be a bounded operator on $L^2(\mu)$ and $m\in \mathbb{N}$. Then $W$ is quasi-m-isometry if and only if
$$G_m=\sum_{k=0}^{m}(-1)^{m-k}\binom{m}{k}J_{k+1}=0,\ \ \ \text{$\mu$-a.e},$$

in which $J_0=1$ and $J_n=h_nE_n(|u_n|^2)\circ\varphi^{-n}=hE(J_{n-1}|u|^2)\circ\varphi^{-1}$.
Moreover, $W$ is m-isometry if and only if 
$$G^0_m=\sum_{k=0}^{m}(-1)^{m-k}\binom{m}{k}J_{k}=0, \ \ \ \text{$\mu$-a.e}.$$
\end{thm}
\begin{proof}
As we know, for each $n\in \mathbb{N}$, we have $W^{n^*}W^n=M_{J_n}$. Hence $W$ is quasi-m-isometry if and only if
$$\sum_{k=0}^{m}(-1)^{m-k}\binom{m}{k}W^{*^{k+1}}W^{k+1}=0$$
iff

$$\sum_{k=0}^{m}(-1)^{m-k}\binom{m}{k}M_{J_{k+1}}=0$$
iff
$$\langle \sum_{k=0}^{m}(-1)^{m-k}\binom{m}{k}M_{J_{k+1}}f,f\rangle=0,$$
 for all $f\in L^2(\mu)$.
So $W$ is quasi-m-isometry if and only if
$$\langle G_mf,f\rangle=\langle \sum_{k=0}^{m}(-1)^{m-k}\binom{m}{k}J_{k+1}f,f\rangle=0,$$
 for all $f\in L^2(\mu)$. This means that $W$ is quasi-$m$-isometry if and only if $G_m=0$, a.e., $\mu$. Also,  $W$ is m-isometry if and only if
$$\sum_{k=0}^{m}(-1)^{m-k}\binom{m}{k}W^{*^{k}}W^{k}=0$$
iff
$$\langle \sum_{k=0}^{m}(-1)^{m-k}\binom{m}{k}M_{J_{k}}f,f\rangle=0,$$
 for all $f\in L^2(\mu)$. Therefore $W$ is $m$-isometry if and only if $G^0_m=0$, a.e., $\mu$.
\end{proof}
\begin{cor}
The composition operator $C_{\varphi}$ is quasi-$m$-isometry if and only if 
$$\sum_{k=0}^{m}(-1)^{m-k}\binom{m}{k}h_{k+1}=0, \ \ \ \text{$\mu$-a.e},$$
and also, $C_{\varphi}$ is m-isometry if and only if
 $$\sum_{k=0}^{m}(-1)^{m-k}\binom{m}{k}h_{k}=0, \ \ \ \text{$\mu$-a.e}.$$
\end{cor}
By using the binomial formula ($(x+y)^m=\sum_{k=0}^{m}\binom{m}{k}x^ky^{m-k}$,  $x,y\in \mathbb{R}$) and taking $\varphi=I$ in Theorem $\ref{t1.6}$ we have conditions under which the multiplication operator $M_u$ is quasi-$m$-isometry or $m$-isometry. 
\begin{cor}
The followings hold for multiplication operator $M_u$ on $L^2(\mu)$.
\begin{itemize}
 \item  
The multiplication operator $M_u$ is quasi-$m$-isometry if and only if 
$$|u(x)|^2(|u(x)|^2-1)^m=\sum_{k=0}^{m}(-1)^{m-k}\binom{m}{k}|u(x)|^{2(k+1)}=0,$$
for almost all $x\in X$ or equivalently $|u|(1-|u|^2)=0$, $\mu$-a.e., if and only if 
  $$\mu(\{x\in X:u(x)\neq 0\}\cap\{x\in X:|u(x)|\neq 1\})=0.$$\\

\item Multiplication operator $M_u$ is $m$-isometry if and only if
 $$(|u(x)|^2-1)^m=\sum_{k=0}^{m}(-1)^{m-k}\binom{m}{k}|u(x)|^{2k}=0,$$ 
 for almost all $x\in X$, or equivalently $|u|=1$, $\mu$-a.e.\\
 \end{itemize}
\end{cor}
Let $T\in \mathcal{B}(\mathcal{H})$ be a normal operator, i.e., $T^*T=TT^*$. Then for every $k\in \mathbb{N}$ we have $T^{*^k}T^k=(T^*T)^k$.
So we have 
$$B_m(T)=\sum_{k=0}^{m}(-1)^{m-k}\binom{m}{k}(T^*T)^{k}=(T^*T-I)^m.$$
It is clear that if $T$ is normal and $m$-isometry, for some $m\in \mathbb{N}$, then it is isometry and consequently it is unitary. Moreover, if $T$ is $p$-hyponormal ($(T^*T)^p\geq(TT^*)^p$) and $T^n$ is normal for some $n$, then $T$ is normal. Therefore if $T$ is $p$-hyponormal, $T^n$ is normal for some $n$ and $T$ is $m$-isometric, then $T$ is unitary. In the following we characterize normal $m$-isometry weighted composition operator $W=uC_{\varphi}$ on the Hilbert space $L^2(\mu)$.

\begin{prop}
Let $W=uC_{\varphi}$ be a normal bounded operator on $L^2(\mu)$ and $m\in \mathbb{N}$. Then $W$ is quasi-m-isometry if and only if
$$J(J-1)^m=0,\ \ \ \text{$\mu$-a.e},$$

in which $J=hE(|u|^2)\circ\varphi^{-1}1$.
Moreover, $W$ is m-isometry if and only if 
$$(J-1)^m=0, \ \ \ \text{$\mu$-a.e}.$$
\end{prop}
\begin{proof} As we computed in the proof of Theorem \ref{t1.6}, $W$ is quasi-m-isometry

$$G_m=\sum_{k=0}^{m}(-1)^{m-k}\binom{m}{k}J^{k+1}=J(J-1)^m=0,\ \ \ \text{$\mu$-a.e},$$

and $W$ is m-isometry if and only if 
$$G^0_m=\sum_{k=0}^{m}(-1)^{m-k}\binom{m}{k}J^{k}=(J-1)^m=0, \ \ \ \text{$\mu$-a.e}.$$
So we get the result.
\end{proof}
If we set $u\equiv $, then $W=uC_{\varphi}=C_{\varphi}$. So have the following corollary for composition operators.
\begin{cor}
Let $C_{\varphi}$ be a normal bounded operator on $L^2(\mu)$ and $m\in \mathbb{N}$. Then\\
\begin{itemize}
  \item  $C_{\varphi}$ is quasi-m-isometry if and only if
$h(h-1)^m=0,\ \ \ \text{$\mu$-a.e}$,
and consequently $C_{\varphi}$ is quasi-m-isometry iff $C_{\varphi}$ is quasi-isometry iff 
$h(h-1)^m=0,\ \ \ \text{$\mu$-a.e}$, 
iff
$h(h-1)=0,\ \ \ \text{$\mu$-a.e}$.
\item  $C_{\varphi}$ is m-isometry if and only if 
$(h-1)^m=0, \ \ \ \text{$\mu$-a.e}$,
and consequently $C_{\varphi}$ is m-isometry iff $C_{\varphi}$ is isometry iff
$(h-1)^m=0, \ \ \ \text{$\mu$-a.e}$, iff $h=1=0, \ \ \ \text{$\mu$-a.e}$.
\end{itemize}
\end{cor}
As is known in the literature, the spectrum of the multiplication operator $M_J$ on $L^2(\mu)$ is the essential range of $J$, i.e, $\sigma(M_J)=\text{essential range}(J)$ . Also, it is known that for every normal bounded linear operator $T$ on a Hilbert space $\mathcal{H}$ we have $\sigma(T^*T)=\{|\lambda|^2:\lambda\in \sigma(T)\}$. Therefore 
$$\{|\lambda|^2:\lambda\in \sigma(W)\}=\sigma(W^*W)=\sigma(M_J)=\text{ess\ range(J)}.$$
So $\sigma(W)=\{\lambda\in \mathbb{C}: |\lambda|^{\frac{1}{2}}\in \text{ess\ range(J)}\}$. By these observations we find that normal-$m$-isometry and normal quasi-$m$-isometry weighted composition operators have finite spectrum.
\begin{cor}
Let $W=uC_{\varphi}$ be a normal bounded operator on $L^2(\mu)$, $\sigma(W)$ be the spectrum of $W$ and $m\in \mathbb{N}$. Then the followings hold:\\
\begin{enumerate}
\item If $W$ is quasi-m-isometry, then $\sigma(W)\subseteq \{0,1\}$ and so $\text{ess\ range(J)}\subseteq \{0,1\}$;

\item If $W$ is m-isometry, then $\sigma(W)\subseteq \{1\}$ and so $\text{ess\ range(J)}\subseteq \{1\}$ ;
\end{enumerate}
 If $W$ is $p$-hyponormal and $W^n$ is normal for some $n$, then (1) and (2) hold for $W$.

\end{cor}

\begin{cor}
Let $C_{\varphi}$ be a normal bounded operator on $L^2(\mu)$, $\sigma(C_{\varphi})$ be the spectrum of $C_{\varphi}$ and $m\in \mathbb{N}$. Then the followings hold:\\
\begin{enumerate}
\item If $C_{\varphi}$ is quasi-m-isometry, then $\sigma(C_{\varphi})\subseteq \{0,1\}$ and so $\text{ess\ range(h)}\subseteq \{0,1\}$;

\item If $C_{\varphi}$ is m-isometry, then $\sigma(C_{\varphi})\subseteq \{1\}$ and so $\text{ess\ range(h)}\subseteq\{1\}$;
\end{enumerate}
If $C_{\varphi}$ is $p$-hyponormal and $C_{\varphi}^n$ is normal for some $n$, then (1) and (2) hold for $C_{\varphi}$.

\end{cor}

In the sequel we prove that if $J$ is bounded away from zero on some measurable set with positive measure, then $J_n$ is so, for all $n\in \mathbb{N}$.
\begin{thm}\label{t1.3}
If there exists a measurable set $A\subseteq S(J)$ with $\mu(A)>0$ and $\varphi^{-1}(A)\subseteq A$, such that $J\geq \delta$ a.e., on $A$ for some $\delta>0$, then for every $n\in\mathbb{N}$, we have $J_n\geq \delta^n$, a.e., on $A$.
\end{thm}
\begin{proof}
As we defined $J_0=1$ and $J_n=h_nE_n(|u_n|^2)\circ\varphi^{-n}=hE(J_{n-1}|u|^2)\circ\varphi^{-1}$. Let $A\in \mathcal{F}$, $A\subseteq S(J)$ with $\mu(A)>0$ and $\varphi^{-1}(A)\subseteq A$, such that $J\geq \delta$ a.e., on $A$ for some $\delta>0$. Hence for every $B\in \mathcal{F}$ such that $B\subseteq A$ we have $\varphi^{-1}(B)\subseteq \varphi^{-1}(A)\subseteq A$ and so 
\begin{align*}
\int_BJ_2d\mu&=\int_X\chi_BhE(J|u|^2)\circ\varphi^{-1}d\mu\\
&\int_X\chi_B\circ\varphi E(J|u|^2)d\mu\\
&=\int_X\chi_{\varphi^{-1}(B)} J|u|^2d\mu\\
&\geq \delta\int_X\chi_{\varphi^{-1}(B)}|u|^2d\mu\\
&=\delta\int_X\chi_Bh(|u|^2)\circ\varphi^{-1}d\mu\\
&\geq \int_B\delta^2d\mu.
\end{align*}
This implies that $J_2\geq \delta^2$, a.e., on $A$. By induction we get that for every $n\in\mathbb{N}$ the inequality $J_n\geq \delta^n$, a.e., on $A$, holds.
\end{proof}
\begin{prop}
If there exists a measurable set $A$ such that $\varphi^{-1}(A_{\delta})\subseteq A_{\delta}$, such that $\mu(A_{\delta}=\{x\in A: \delta\leq J(x)<1\})>0$, for some $\delta>0$, then $W$ is not 2-isometry. 
\end{prop}
\begin{proof} Since $J\geq \delta$ a.e., on $A_{\delta}$ and $\varphi^{-1}(A_{\delta})\subseteq A_{\delta}$, then by Theorem \ref{t1.3} we get that $J_2\geq\delta^2$ a.e., on $A_{\delta}$. Hence for every $x\in A_{\delta}$ we have
$$J_2(x)-2J_1(x)+1\geq \delta^2-2J_1(x)+1=(\delta-J_1(x))^2+1-J^2_1(x).$$

So we have 
$$(\delta-J_1(x))^2+1-J^2_1(x)\leq J_2(x)-2J_1(x)+1.$$
Since $\mu(A_{\delta})>0$, then we get that $J_2(x)-2J_1(x)+1\neq 0$, $\mu$, a.e., and so $W$ is not 2-isometry.\\
\end{proof}

\begin{exam} (a) Let $X=[0, 1]$, $d\mu=dx$ and $\Sigma$
be the Lebesgue sets. Take $u(x)=\sqrt{x(x^2+1)}$ and $\varphi(x)=x^2$. 
 It follows that
$\varphi^{-1}(x)=\sqrt{x}$, $h(x)=\frac{1}{2\sqrt{x}}$ and for every $a,b\in [0,1]$ we have 
 $$\int_{\varphi^{-1}(a,b)}E(f)(x)dx=\int_{\varphi^{-1}(a,b)}f(x)dx=\int_{\sqrt{a}}^{\sqrt{b}}f(x)dx=\int_{a}^{b}\frac{f(\sqrt{x})}{2\sqrt{x}}dx.$$
 This means that $h(x)E(f)\circ\varphi^{-1}(x)=\frac{f(\sqrt{x})}{2\sqrt{x}}$. Hence 
 $$J(x)=h(x)E(|u|^2)\circ\varphi^{-1}(x)=\frac{x+1}{2}.$$
 Also, we have 
 $$J_2(x)=h(x)E(J_1|u|^2)\circ\varphi^{-1}(x)=\frac{(\sqrt{x}+1)(x+1)}{4}.$$
 inductively we get that for every $k\geq 2$, 
 $$J_k=\frac{(\sqrt{x}+1)(x+1)^{k-1}}{2^k}.$$
 Direct computations shows that 
 $$J_2(x)-2J_1(x)+1= \frac{(\sqrt{x}+1)(x-4\sqrt{x}+2)}{4}$$
 and clearly $J_2(x)-2J_1(x)+1\neq 0$, for all $x\in X$.
 Therefore the weighted composition operator $W=uC_{\varphi}$ is not $2$-isometric on the Hilbert space $L^2(\mu)$.\\
 (b) 
 Equip with the assumptions of part (a), if we take $u(x)=\sqrt{[x]}$, in which $[x]$ means the integer part of $x$, then we have 
  $$J(x)=h(x)E(|u|^2)\circ\varphi^{-1}(x)=\frac{[\sqrt{x} ]}{2\sqrt{x}}.$$
  It is clear that $J=0$  $\mu$ a.e., on $X$. Consequently, $J_n=0$  $\mu$ a.e., on $X$, for each $n\in \mathbb{N}$. This implies that $W^*B_mW=0$ and so $W$ is $m$-quasi isometric for all $m\in \mathbb{N}$, but $B_m=(-1)^mI$ and therefore $W$ is not $m$-isometric any $m\in \mathbb{N}$.\\

(c) If we replace $u(x)=\sqrt{x([x^2+1])}$ by 
$u(x)=\frac{\sqrt{x([x^2+1])}}{2}$ in the part (c), then we have 
$$J_n(x)=\left\{
  \begin{array}{ll}
    \frac{1}{2^{2n}}, & \hbox{$0\leq x<1$;} \\
    \frac{1}{2^{n}}, & \hbox{$x=1$.}
  \end{array}
\right.$$
and so we get that 
$$G_m=\frac{1}{4}\sum_{k=0}^{m}(-1)^{k}\binom{m}{k}(\frac{1}{2})^{2k}=\frac{1}{4}(\frac{3}{4})^m,\ \ \ \text{$\mu$-a.e},$$
and 
$$G^0_m=\sum_{k=0}^{m}(-1)^{k}\binom{m}{k}(\frac{1}{2})^{2k}=(\frac{3}{4})^m, \ \ \ \text{$\mu$-a.e}.$$
In this case the weighted composition operator $W$ is not $m$-quasi isometric and is not $m$-isometric, for any $m\in \mathbb{N}$.\\

(d)
Let $m = \{m_n\}_{n=1}^{\infty}$ be a sequence of positive real numbers. Consider the space $l^2(m) = L^2(\mathbb{N}, 2^{\mathbb{N}}, \mu)$, where $2^{\mathbb{N}}$ is the power set of natural numbers and $\mu$ is a measure on $2^{\mathbb{N}}$ defined by $\mu(\{n\}) = m_n$. Let $u = \{u_n\}_{n=1}^{\infty}$ be a sequence of non-negative real numbers. Let $\varphi: \mathbb{N} \to \mathbb{N}$ be a non-singular measurable transformation, i.e. $\mu \circ \varphi^{-1} \ll \mu$. Direct computation shows that
		
		\qquad $h(k) = \frac{1}{m_k} \sum_{j \in \varphi^{-1}(k)} m_j$, \quad $E_{\varphi}(f)(k) = \frac{\sum_{j \in \varphi^{-1}(k)} f_j m_j}{\sum_{j \in \varphi^{-1}(k)} m_j}$,
		for all non-negative sequence $f = \{f_n\}_{n=1}^{\infty}$ and $k \in \mathbb{N}$. For each $j \in \mathbb{N}$ we have 
		$$
		J_{1}(j)=\frac{1}{m_j} \sum_{i \in \varphi^{-1}(j)} u(i)^2 m_i,
		$$
and so  
$$J_{2}(j) = \frac{1}{m_j} \sum_{i \in \varphi^{-1}(j)} \frac{1}{m_i} \sum_{k \in \varphi^{-1}(i)} u(i)^2 u(k)^2m_i m_k.$$
By these observation we get that $W=uC_{\varphi}$ is 2-isometry on $l^2(m)$ if and only if 
$$\frac{1}{m_j} \sum_{i \in \varphi^{-1}(j)} \frac{1}{m_i} \sum_{k \in \varphi^{-1}(i)} u(i)^2 u(k)^2m_i m_k-2\frac{1}{m_j} \sum_{i \in \varphi^{-1}(j)} u(i)^2 m_i+1=0.$$
\end{exam}

\end{document}